\documentclass{iopart}

\usepackage{iopams}
 \expandafter\let\csname equation*\endcsname\relax
  \expandafter\let\csname endequation*\endcsname\relax
\usepackage{amsmath}
\usepackage{amsthm}
\newtheorem{theorem}{Theorem}
\newtheorem{lemma}{Lemma}

\begin{document}
\title{Multiple Meixner polynomials and non-Hermitian oscillator Hamiltonians}
\author{F Ndayiragije$^1$ and W Van Assche$^2$} 
\address{$^1$ D\'epartement de Math\'ematiques, Universit\'e du Burundi, Campus Mutanga, P.O. BOX 2700, Bujumbura, Burundi}
\address{$^2$ Department of Mathematics, KU Leuven, Celestijnenlaan 200B box 2400, BE-3001 Leuven, Belgium}
\ead{ndayiragijefrancois@yahoo.fr}
\ead{walter@wis.kuleuven.be}

\begin{abstract}
Multiple Meixner polynomials are polynomials in one variable which satisfy orthogonality relations
with respect to $r>1$ different negative binomial distributions (Pascal distributions).
There are two kinds of multiple Meixner polynomials, depending on the selection 
of the parameters in the negative binomial distribution. We recall their definition and some formulas
and give generating functions and explicit expressions for the coefficients in the nearest neighbor recurrence relation. Following a recent construction of Miki, Tsujimoto, Vinet and Zhedanov (for multiple Meixner polynomials of the first kind), we construct $r>1$ non-Hermitian oscillator Hamiltonians in 
$r$ dimensions which are
simultaneously diagonalizable and for which the common eigenstates are expressed in terms of multiple Meixner polynomials of the second kind.   
\end{abstract}

\pacs{02.30.Gp, 03.65.Ge}
\ams{42C05, 33C47, 81R05}
\submitto{J. Phys. A: Math. Theor.}

\section{Introduction}
Meixner polynomials $M_n(x;\beta, c)$, $n= 0,1,2,\ldots$ are orthogonal
polynomials for the negative binomial distribution (Pascal
distribution), i.e.,
\[  \sum_{k=0}^\infty M_n(k;\beta,c)M_m(k;\beta,c) \frac {(\beta
)_k}{k!} c^k =0, \qquad n\neq m, \] 
where $\beta > 0$ and $0<c<1$.
These polynomials are orthogonal on the positive integers.
In Chihara \cite[Chap.VI~, \S 3]{Chihara} they are called Meixner polynomials of the first kind,
but nowadays they are more commonly known as Meixner polynomials (see, e.g., \cite[\S 6.1]{Ismail} and \cite[\S 9.10]{Koekoek}) 
and the Meixner polynomials of the second kind in \cite[Chap.VI~, \S 3]{Chihara} are now known as Meixner-Pollaczek polynomials 
(see, e.g., \cite[\S 5.9]{Ismail} and \cite[\S 9.7]{Koekoek}).
The Meixner polynomials are explicitly given by
\[    \tilde{M}_n(x;\beta,c) = {}_2F_1\left(  \left.\begin{array}{c} -n, -x \\ \beta \end{array}\right|  1 - \frac1{c} \right) 
         = \sum_{k=0}^n \binom{n}{k} \frac{(-x)_k}{(\beta)_k} \left(\frac{1-c}{c} \right)^k \]
so that the leading coefficient is $(1-1/c)^n/(\beta)_n$. The monic polynomials are given by 
$M_n(x;\beta,c) =(\beta)_nc^n/(c-1)^n \tilde{M}_n(x;\beta,c)$.

We will investigate multiple Meixner  polynomials, which are
polynomials of one variable with orthogonality properties with
respect to more than one negative binomial distribution. 
These multiple Meixner polynomials will be used in Section \ref{oscillator} to construct non-Hermitian oscillator Hamiltonians
in $r$ dimensions which are simultaneously diagonalizable and for which the eigenstates are expressed in terms of multiple Meixner polynomials.
This construction was first done in \cite{MVZ} for multiple Charlier polynomials and in \cite{MTVZ} for multiple Meixner polynomials
of the first kind, but there is a gap in their construction which we will fix in Section \ref{MM1}. We extend this construction to multiple
Meixner polynomials of the second kind in Section \ref{MM2}. The multiple Meixner polynomials hence give a mathematical approach to a physical 
phenomenon involving a chain of Hamiltonians.
 
There are two kinds of multiple Meixner polynomials, depending on how one chooses the negative binomial
distributions (\cite{ACV}, \cite {Ismail}). 
Let $\vec n = (n_1,\ldots,n_r)$ 
be a multi-index of size $|\vec n| = n_1+\cdots+n_r$, then the multiple Meixner polynomials of the first kind $M_{\vec{n}}^{(1)}$
are the monic polynomials of degree $|\vec n |$ for which 
\[    \sum_{k=0}^\infty M_{\vec{n}}^{(1)}(k) k^\ell \frac{(\beta)_k(c_j)^k}{k!} = 0, \qquad \ell=0,1,\ldots,n_j-1, \ j=1\ldots,r, \]
where $\beta > 0$ and $0 < c_i \neq c_j < 1$ whenever $i \neq j$.
The multiple Meixner polynomials of the second kind $M_{\vec{n}}^{(2)}$ are the monic polynomials of degree $|\vec{n}|$
for which 
\[    \sum_{k=0}^\infty M_{\vec{n}}^{(2)}(k) k^\ell \frac{(\beta_j)_kc^k}{k!} = 0, \qquad \ell=0,1,\ldots,n_j-1, \ j=1\ldots,r, \]
where $0<c<1$ and $\beta_j >0$ for all $j=1,\ldots,r$, and $\beta_i-\beta_j \notin \mathbb{Z}$ when $i \neq j$.
For $r=1$ we retrieve the Meixner polynomials. 
An explicit formula for the multiple Meixner
polynomials of the first kind can be obtained using the Rodrigues
formula \cite{ACV, Ismail},
\begin{equation}  \label{Rodrigues 1}
     M_{\vec n}^{(1)}(x) =(\beta)_{|\vec n|}\prod_{j=1}^r  \left( \frac{c_j}{c_j-1}
    \right)^{n_j} \frac {\Gamma(\beta)\Gamma (x+1)}{\Gamma (\beta +
    x)} \left( \prod_{j=1}^r  c_j^{-x}\nabla^{n_j}c_j^x \right)
     \frac {\Gamma (|\vec n|+\beta +x)}{\Gamma (|\vec n|+ \beta )\Gamma (x+1)}
  \end{equation}
where $\nabla$ is the backward difference operator, given by $\nabla f(x) = f(x)-f(x-1)$.
For this backward difference operator $\nabla$ we have the
property
\begin{equation} \label{backward}
   \nabla^n f(x) = \sum_{k=0}^n \binom{n}{k} (-1)^k f(x-k) .
\end{equation}
Combining the Rodrigues formula \eref{Rodrigues 1} and
\eref{backward} we have 
\begin{multline*}
M_{\vec n}^{(1)}(x) = \sum_{k_1=0}^{n_1} \cdots \sum_{k_r=0}^{n_r}
   \binom{n_1}{k_1} \cdots \binom{n_r}{k_r} \frac {c_1^{n_1-k_1} \ldots c_r^{n_r- k_r}}{(c_1-1)^{n_1} \ldots
   (c_r-1)^{n_r}} (-1)^{k_1+ \cdots +k_r} \\
    \times  \frac {\Gamma (x+1)}{\Gamma (x-|\vec k|+1)} \quad  \frac{\Gamma (\beta + |\vec n|+x- |\vec k|)}{\Gamma (\beta +x)}
\end{multline*}
which gives an explicit expression for the multiple Meixner
polynomials of the first kind
\begin{equation} \label{explicit 1}
M_{\vec n}^{(1)}(x)= \sum_{k_1=0}^{n_1} \cdots \sum_{k_r=0}^{n_r}
   \binom{n_1}{k_1} \cdots \binom{n_r}{k_r} (-x)_{|\vec
   k|} \prod_{j=1}^r \frac{c_j^{n_j-k_j}}{(c_j-1)^{n_j}}
   (\beta+x)_{|\vec n|-|\vec k|}.
\end{equation}
These multiple Meixner polynomials are related to the multiple Charlier polynomials by the asymptotic relation
\[   \lim_{\beta \to \infty} M_{\vec{n}}^{(1)}(x;\beta,c_1=a_1/\beta,\cdots,c_r=a_r/\beta) = C_{\vec{n}}(x;a_1,\ldots,a_r),  \]
where the multiple Charlier polynomial $C_{\vec{n}}$ is the monic polynomial of degree $|\vec{n}|$ satisfying
\[   \sum_{k=0}^\infty C_{\vec{n}}(k) k^\ell \frac{a_j^k}{k!} = 0, \qquad \ell=0,1,\ldots,n_j-1, \ j=1,\ldots,r, \]
where $a_1,\ldots,a_r >0$ and $a_i \neq a_j$ whenever $i \neq j$ (see \cite{ACV}, \cite{Ismail}).
 
An explicit expression for the multiple Meixner polynomials of the second kind can be obtained
using the Rodrigues formula \cite{ACV, Ismail},
\begin{multline}  \label{Rodrigues 2}
     M_{\vec n}^{(2)}(x) = \left ( \frac{c}{c-1} \right )^{|\vec n|}
     \prod_{j=1}^r (\beta_j)_{n_j} \frac{\Gamma (x+1)}{c^x}  \\
    \times  \left( \prod_{j=1}^r  \frac {\Gamma (\beta_j)}{\Gamma (\beta_j +
     x)} \nabla^{n_j} \frac {\Gamma (\beta_j+n_j + x)}{\Gamma (\beta_j +
     n_j)} \right) \frac {c^x}{\Gamma (x+1)} .
  \end{multline}
Now we can  again use property  \eref{backward} to find
\begin{align*}
   M_{\vec n}^{(2)}(x) &=\prod_{k=1}^r (\beta_k)_{n_k}\sum_{k_1=0}^{n_1} \cdots \sum_{k_r=0}^{n_r}
   \binom{n_1}{k_1} \cdots \binom{n_r}{k_r} \frac {c^{|\vec n|-|\vec k|}}{(c-1)^{|\vec n|}} (-1)^{|\vec k|}\frac{\Gamma (x+1)}{\Gamma \Big((x+1)-|\vec n|
   \Big)}\\
   & \quad \times \frac {\Gamma (\beta_1)}{\Gamma (\beta_1+x)}   \frac {\Gamma (\beta_1+n_1+x-k_1)}{\Gamma (\beta_1+n_1)}
   \frac   {\Gamma (\beta_2)}{\Gamma (\beta_2+x-k_1)} \frac {\Gamma (\beta_2+n_2+x-k_1-k_2)}{\Gamma
   (\beta_2+n_2)} \\
   & \quad \times  \cdots \frac{\Gamma (\beta_r)}{\Gamma (\beta_r+x-k_1-k_2- \cdots
   -k_{r-1})} \frac {\Gamma (\beta_r+n_r+x-k_1- \cdots -k_r)}{\Gamma (\beta_r+
   n_r)}
\end{align*}
which gives an explicit expression for the multiple Meixner
polynomials of the second kind
\begin{equation} \label{explicit 2}
M_{\vec n}^{(2)}(x)= \sum_{k_1=0}^{n_1} \cdots \sum_{k_r=0}^{n_r}
   \binom{n_1}{k_1} \cdots \binom{n_r}{k_r} \frac {c^{|\vec n|-|\vec k|}}{(c-1)^{|\vec
   n|}} (-x)_{|\vec k|} \prod_{j=1}^r \left (
   \beta_j+x-\sum_{i=1}^{j-1}k_i \right )_{n_j-k_j}.
\end{equation}
These multiple Meixner polynomials are related to the multiple Charlier polynomials by the asymptotic relation
\[   \lim_{\beta \to \infty} M_{\vec{n}}^{(2)}(x;\beta_1=a_1\beta,\ldots,\beta_r=a_r\beta, c=1/\beta) = C_{\vec{n}}(x;a_1,\ldots,a_r).  \]

We will use the explicit expressions \eref{explicit 1} and \eref{explicit 2} for obtaining the generating functions in Section \ref{3.2}, 
for computing the recurrence coefficients in Section \ref{3.3}, and for proving the square summability in the appendix.

\section{Generating function}  \label{3.2}

\subsection{Multiple Meixner polynomials of the first kind}
Meixner polynomials have the generating function
\cite[Eq.~(6.1.8)]{Ismail}, \cite[Eq.~(9.10.11)]{Koekoek}
\begin{equation}  \label{GenMeix1}
\sum_{n=0}^\infty \tilde{M}_n(x;\beta,c) \frac{(\beta)_n}{n!} t^n =
\left(1-\frac{t}{c} \right)^x (1-t)^{-x-\beta}.
\end{equation}
For multiple Meixner polynomials of the first kind one has a
multivariate generating function (with $r$ variables). It was found in \cite[\S 9]{MTVZ} where it was proved using the Bargmann realization on the Lie algebra $W(r)$ made out of $r$ copies of the Heisenberg-Weyl algebra (see also Section \ref{oscillator}).

\begin{theorem}  \label{thm:gen1}
Multiple Meixner polynomials of the first kind have the following
(multivariate) generating function
\begin{multline}  \label{generating 1}
   \sum_{n_1=0}^\infty \cdots \sum_{n_r=0}^\infty  M_{\vec{n}}^{(1)}(x) \frac{t_1^{n_1}  \cdots
    t_r^{n_r}}{n_1!  \cdots n_r!}  \\
  = \left(1+\frac{t_1}{1-c_1}+\cdots+\frac{t_r}{1-c_r}\right)^x
  \left(1+ \frac{c_1}{1-c_1}t_1 +\cdots+\frac{c_r}{1-c_r}t_r\right)^{-x-\beta}.
\end{multline}
\end{theorem}
We will use the following lemma, which is basically an alternative way to express the multinomial theorem \cite[Eq. (220) on p.~329]{SK}.
\begin{lemma} \label{lemma:Meix}
The generating function for the multinomial coefficients is
\begin{equation}  \label{lemmaMeix}
\sum_{n_1=0}^\infty \cdots \sum_{n_r=0}^\infty \frac{(-x)_{|\vec
n|}}{n_1! \cdots n_r!}s_1^{n_1} \cdots s_r^{n_r} = (1-s_1- \cdots
-s_r)^x.
\end{equation}
This series converges absolutely and uniformly for $|s_1|+\cdots+|s_r| < 1$ when $x \notin \mathbb{N}$ and
contains a finite number of terms if $x \in \mathbb{N}$.
\end{lemma}
\begin{proof}
The use of the binomial theorem and the multinomial theorem leads to
\begin{align*}
   (1-s_1- \cdots -s_r)^x  &=\sum_{n=0}^{\infty}\binom {x}{n} (-1)^n
    (s_1+\cdots+ s_r)^n \\
   &=\sum_{n_1=0}^{\infty} \cdots \sum_{n_r=0}^{\infty} \binom {|\vec n|}{n_1, \cdots,
   n_r} s_1^{n_1} \cdots s_r^{n_r} (-1)^{|\vec n|} \binom {x}{|\vec n|} \\
   &=\sum_{n_1=0}^\infty \cdots \sum_{n_r=0}^{n_r}\frac{(-x)_{|\vec
n|}}{n_1! \cdots n_r!}s_1^{n_1} \cdots s_r^{n_r} .
\end{align*}
\end{proof}

\begin{proof}[Proof of Theorem \ref{thm:gen1}]
Using (\ref{explicit 1}) the multivariate generating function is
\begin{multline*}
\sum_{n_1=0}^{\infty}\cdots\sum_{n_r=0}^{\infty}\sum_{k_1=0}^{n_1}\cdots\sum_{k_r=0}^{n_r}
\frac{(-x)_{|\vec k|}}{k_1!(n_1-k_1)! \cdots k_r!(n_r-k_r)!} \\
\times \frac{c_1^{n_1-k_1} \cdots c_r^{n_r-k_r}}{(c_1-1)^{n_1} \cdots
(c_r-1)^{n_r}}   (\beta+x)_{|\vec n|-|\vec k|}t_1^{n_1} \cdots t_r^{n_r}  .
\end{multline*}
Changing the order of summation gives
\begin{multline*}
\sum_{k_1=0}^{\infty}\cdots\sum_{k_r=0}^{\infty}\sum_{n_1=k_1}^{\infty}\cdots\sum_{n_r=k_r}^{\infty}
\frac{(-x)_{|\vec k|}}{k_1!(n_1-k_1)! \cdots k_r!(n_r-k_r)!} \\
\times \frac{c_1^{n_1-k_1} \cdots c_r^{n_r-k_r}}{(c_1-1)^{n_1} \cdots
(c_r-1)^{n_r}}
 (\beta+x)_{|\vec n|-|\vec k|}t_1^{n_1} \cdots t_r^{n_r} ,
\end{multline*}
and by setting $\ell_i=n_i-k_i$ and putting the factors in $\ell_i$ and $k_i$ together
\begin{multline*}
 \left(\sum_{k_1=0}^{\infty}\cdots\sum_{k_r=0}^{\infty}\frac{(-x)_{|\vec k|}}{k_1! \cdots
k_r!} \bigg (\frac{t_1}{c_1-1}\bigg)^{k_1} \cdots
\bigg(\frac{t_r}{c_r-1}\bigg)^{k_r}\right)\\
\times\left(\sum_{l_1=0}^{\infty}\cdots\sum_{l_r=0}^{\infty}\frac{(\beta+x)_{|\vec
l|}}{l_1! \cdots l_r!}  \bigg (\frac{c_1}{c_1-1}t_1 \bigg )^{l_1}
\cdots \bigg(\frac{c_r}{c_r-1}t_r \bigg )^{l_r} \right) .
\end{multline*}
Now we use Lemma \ref{lemma:Meix} to obtain the desired result.
\end{proof}

\subsection{Multiple Meixner polynomials of the second kind}
For multiple Meixner polynomials of the second kind one has the following
multivariate generating function (with $r$ variables).

\begin{theorem}  \label{thm:gen2}
Multiple Meixner polynomials of the second kind have the following
(multivariate) generating function
\begin{multline}  \label{generating 2}
   \sum_{n_1=0}^\infty \cdots \sum_{n_r=0}^\infty  M_{\vec{n}}^{(2)}(x) \frac{t_1^{n_1}  \cdots
    t_r^{n_r}}{n_1!  \cdots n_r!} \\
    =\left ( 1-\frac{1}{c}
    \bigg[1-\prod_{j=1}^r \bigg (1+\frac{c}{1-c}t_j \bigg )\bigg]
    \right)^x \prod_{j=1}^r \left (1+\frac{c}{1-c}t_j
    \right )^{-x-\beta_j} .
\end{multline}
\end{theorem}
\begin{proof}
If we denote the left hand side in \eref{generating 2} by G, then using \eref{explicit 2} the multivariate generating function is
\begin{multline*}
G = \sum_{n_1=0}^{\infty}\cdots\sum_{n_r=0}^{\infty}\sum_{k_1=0}^{n_1}\cdots\sum_{k_r=0}^{n_r}
\frac{(-x)_{|\vec k|}}{k_1!(n_1-k_1)! \cdots k_r!(n_r-k_r)!}
\frac{c^{|\vec n|-|\vec k|}}{(c-1)^{|\vec n|}}
(\beta_1+x)_{n_1-k_1}\\
 \times (\beta_2+x-k_1)_{n_2-k_2}(\beta_r+x-k_1- \cdots -k_{r-1})_{n_r-k_r}t_1^{n_1} \cdots
 t_r^{n_r} \cdot
\end{multline*}
Changing the order of summation gives
\begin{multline*}
G = \sum_{k_1=0}^{\infty}\cdots\sum_{k_r=0}^{\infty}\sum_{n_1=k_1}^{\infty}\cdots\sum_{n_r=k_r}^{\infty}
\frac{(-x)_{|\vec k|}}{k_1!(n_1-k_1)! \cdots k_r!(n_r-k_r)!}
\frac{c^{|\vec n|-|\vec k|}}{(c-1)^{|\vec n|}}
(\beta_1+x)_{n_1-k_1}\\
 \times (\beta_2+x-k_1)_{n_2-k_2}(\beta_r+x-k_1- \cdots -k_{r-1})_{n_r-k_r}t_1^{n_1} \cdots
 t_r^{n_r} ,
\end{multline*}
and by setting $\ell_i=n_i-k_i$ and putting the factors in $\ell_i$ and $k_i$ together
\begin{multline*}
 G =  \sum_{k_1=0}^{\infty}\cdots \sum_{k_r=0}^\infty \frac{(-x)_{|\vec{k}|}}{k_1!\cdots k_r!} \bigg(\frac{t_1}{c-1} \bigg)^{k_1}\cdots
  \bigg(\frac{t_r}{c-1} \bigg)^{k_r}
  \sum_{l_1=0}^{\infty}\frac{(\beta_1+x)_{\ell_1}}{\ell_1!}\bigg(\frac{c}{c-1}t_1
 \bigg)^{\ell_1}\\
 \times
 \sum_{\ell_2=0}^{\infty}\frac{(\beta_2+x-k_1)_{\ell_2}}{\ell_2!}\bigg(\frac{c}{c-1}t_2
\bigg)^{\ell_2} \cdots \sum_{\ell_r=0}^{\infty}\frac{(\beta_1+x-k_1-
\cdots -k_{r-1})_{\ell_r}}{\ell_r!}\bigg(\frac{c}{c-1}t_r
\bigg)^{\ell_r} .
\end{multline*}
For each sum involving the $\ell_i$ we can use the binomial theorem
\[  \sum_{\ell=0}^\infty \frac{(\beta)_\ell}{\ell!} t^\ell = (1-t)^{-\beta}, \qquad |t| < 1, \]
to find
\begin{multline*}
G = \sum_{k_1=0}^{\infty} \cdots \sum_{k_r=0}^\infty \frac{(-x)_{|\vec{k}}}{k_1!\cdots k_r!}
\left(\frac{t_1}{c-1}\bigg(1+\frac{c}{1-c}t_2 \bigg) \cdots
\bigg(1+\frac{c}{1-c}t_r \bigg)\right)^{k_1}  \\
\times \left( \frac{t_2}{c-1}\bigg(1+\frac{c}{1-c}t_3 \bigg) \cdots
\bigg(1+\frac{c}{1-c}t_r \bigg)\right )^{k_2} \cdots
    \left( \frac{t_{r-1}}{c-1}\bigg(1+\frac{c}{1-c}t_r \bigg) \right )^{k_{r-1}}
    \left( \frac{t_r}{c-1}\right )^{k_r}\\
\times \prod_{j=1}^r\left (1+\frac{c}{1-c}t_j \right)^{-x-\beta_j} .
\end{multline*}
If we let $a=c/(1-c)$ and
\[   s_j = t_j(1+at_{j+1})\cdots(1+at_t), \qquad 1 \leq j \leq r, \]
then this simplifies to
\[   G = \sum_{k_1=0}^\infty \cdots \sum_{k_r=0}^\infty \frac{(-x)_{|\vec{k}|}}{k_1!\cdots k_r!} s_1^{k_1}\cdots s_r^{k_r} 
  \left( \frac{-a}{c} \right)^{|\vec{k}|} \prod_{j=1}^r \left (1+\frac{c}{1-c}t_j     \right )^{-x-\beta_j}, \]
and by using Lemma \ref{lemma:Meix} we then get
\[   G = \left(1+ \frac{a}{c} (s_1+\cdots + s_r) \right)^x \prod_{j=1}^r \left (1+\frac{c}{1-c}t_j \right )^{-x-\beta_j}.  \]
The result then follows by observing that
\begin{align*}
& s_1+s_2+\cdots+s_r \\
& = t_1(1+at_2) \cdots(1+at_r)+t_2(1+at_3)\cdots(1+at_r)+ \cdots + t_{r-1}(1+at_r)+t_r  \\
& = t_1+t_2+\cdots+t_r+a\sum_{i<j}t_it_j+a^2\sum_{i<j<k}t_it_jt_k+\cdots+a^{r-1}t_1t_2 \cdots t_r \\
& =\frac{-1}{a}\bigg(1- \bigg[ 1 + a(t_1+ \cdots+t_r)+ a^2 \sum_{i<j}t_it_j + 
   a^3 \sum_{i<j<k}t_it_jt_k+ \cdots +a^r t_1\cdots t_r \bigg] \bigg)\\
& =\frac{-1}{a}\bigg( 1-(1+at_1)(1+at_2)\cdots(1+at_r) \bigg),
\end{align*}
which implies
\begin{equation*}
1+ \frac{a}{c} (s_1+s_2+\cdots+s_r) = 1-\frac{1}{c}
    \bigg[ 1-\prod_{j=1}^r \bigg (1+\frac{c}{1-c}t_j \bigg )\bigg] .
\end{equation*}
\end{proof}

\section{Recurrence relations} \label{3.3}
For multiple orthogonal polynomials
there are always nearest neighbor recurrence relations
 of the form
\begin{equation}  \label{eq:recur}
   xP_{\vec{n}}(x) = P_{\vec{n}+\vec{e_k}}(x) + b_{\vec{n},k} P_{\vec{n}}(x)
    + \sum_{j=1}^r a_{\vec{n},j} P_{\vec{n}-\vec{e}_j}(x),
\end{equation}
where $k\in \{1,\ldots,r\}$ \cite[Thm.~23.1.11]{Ismail}, \cite{WVA}, and
$\vec{e}_k = (0,\ldots,0,1,0,\ldots,0)$ is the $k$th unit vector in
$\mathbb{N}^r$.
 The recurrence relations for multiple Meixner polynomials (of the first and second kind)
were given in \cite{HVA} without proof. Here we will work out the
details of the computations.

\subsection{Multiple Meixner polynomials of the first kind}
\begin{theorem}  \label{thm:rec1}
The nearest neighbor recurrence relations for multiple Meixner
polynomials of the first kind are
\begin{multline}  \label{Meix1:recur1}
   x M_{\vec{n}}^{(1)}(x) = M_{\vec{n}+\vec{e}_k}^{(1)}(x) +\Bigg( (\beta + |\vec n|)\frac{c_k}{1-c_k}+
   \sum_{i=1}^r \frac{n_i}{1-c_i} \Bigg )M_{\vec{n}}^{(1)}(x) \\ 
  + \sum_{j=1}^r \frac{c_j n_j}{(1-c_j)^2}(\beta+|\vec n|-1) M_{\vec{n}-\vec{e}_j}^{(1)}(x).
\end{multline}
\end{theorem}

\begin{proof}
From (\ref{explicit 1}) and
\begin{multline*}
(-x)_{|\vec k|}(\beta +x)_{|\vec n|-|\vec k|}  \\ = (-1)^{|\vec k|}
x^{|\vec n|} + \Bigg ( (-1)^{|\vec k|-1} \binom {|\vec k|}{2}+
(-1)^{|\vec k|}\Big ( \beta + \frac {(2 \beta + |\vec n|-|\vec
k|)(|\vec n|-|\vec k|-1)}{2} \Big ) \Bigg ) x^{|\vec n|-1}+ \cdots
\end{multline*}
we find that  \[M_{\vec n}^{(1)}(x)= x^{|\vec n|}+ \delta _n
x^{|\vec n|-1}+ \cdots , \] where
\begin{multline*}
\delta_{\vec n}= \prod _{j=1}^r \Bigg (\frac {c_j}{c_j-1} \Bigg
)^{n_j} \sum_{k_1=0}^{n_1} \cdots \sum_{k_r=0}^{n_r} \binom{n_1}{k_1}
\cdots \binom {n_r}{k_r} \\ 
 \times \prod_{j=1}^r c_j^{-k_j}(-1)^{|\vec
k|}\Bigg (-\frac{|\vec k|^2}{2}+\frac {|\vec k|}{2}+\frac {2 \beta
+(2 \beta + |\vec n|-|\vec k|)(|\vec n|-|\vec k|-1)}{2} \Bigg ),
\end{multline*}
which leads, with some calculations using the binomial theorem, to
\[ \delta_{\vec n}= \frac {|\vec n|^2+|\vec n|(2 \beta -1)}{2}+(\beta+|\vec n|-1) \sum_{i=1}^r \frac{n_i}{c_i-1} .   \]
If we compare the coefficients of $x^{|\vec n|}$ in (\ref{eq:recur}),
then $ b_{\vec n, k}= \delta_{\vec n}- \delta_{\vec n+ \vec e_k}$,
which for the multiple Meixner polynomials of the first kind gives
\begin{equation}\label{Meix1:coef1}
  b_{\vec n,k}= (\beta +|\vec n|)\frac{c_k}{1-c_k}+ \sum_{i=1}^r
\frac{n_i}{1-c_i} .
\end{equation}
For the recurrence coefficients $a_{\vec n,j}$ we can use \cite[Thm.~3.2, Eq. (3.12)]{WVA}
\[ \frac{a_{\vec n,j}}{a_{\vec n+\vec e_i,j}} =\frac{b_{\vec n-\vec e_j,j}-b_{\vec n-\vec e_j,i}}{b_{\vec n,j}-b_{\vec
n,i}} . \] 
Using (\ref{Meix1:coef1}) we find that
\[  a_{\vec n,j}= \frac {\beta+|\vec n|-1}{\beta+|\vec n|-2} a_{\vec n-\vec e_i,j}.  \] 
If we use this formula recursively to go from $\vec n$ to $n_j \vec e_j$ by decreasing one of the indices (except the $j$th index)
by one, then we arrive at
\[ a_{\vec n,j}= \frac {\beta+|\vec n|-1}{\beta+n_j-1} a_{n_j \vec e_j, j}, \]
and $a_{n_j \vec e_j,j}$ is the recurrence coefficient $a_{n_j}^2$ for the Meixner polynomials $M_n(x;\beta,c_j)$,
which is equal to $\frac{c_j n_j}{(1-c_j)^2}(\beta+n_j-1)$. This gives
\[ a_{\vec n,j}= \frac{c_j n_j}{(1-c_j)^2} (\beta+|\vec n|-1) . \]
\end{proof}

\subsection{Multiple Meixner polynomials of the second kind}

\begin{theorem}  \label{thm:rec2}
The nearest neighbor recurrence relations for multiple Meixner polynomials of the second kind are
\begin{multline}  \label{Meix2:recur1}
  x M_{\vec{n}}^{(2)}(x) = M_{\vec{n}+\vec{e}_k}^{(2)}(x) +\Bigg( \frac{c}{1-c} (n_k+\beta_k)+
    \frac{|\vec n|}{1-c} \Bigg) M_{\vec{n}}^{(2)}(x) \\
    + \sum_{j=1}^r  \frac{cn_j(n_j+\beta_j-1)}{(1-c)^2} \prod_{i\neq j}^r \frac{n_j+\beta_j-\beta_i}{n_j+\beta_j-n_i-\beta_i} 
M_{\vec{n}-\vec{e}_j}^{(2)}(x).
\end{multline}
\end{theorem}

\begin{proof}
From (\ref{explicit 2}) and
\begin{multline*}
(-x)_{|\vec k|}\prod_{j=1}^r\bigg(\beta_j +x-\sum_{i=1}^{j-1}k_i\bigg )_{n_j-k_j}  \\
= (-1)^{|\vec k|} x^{|\vec n|} + (-1)^{|\vec k|}x^{|\vec n|-1}\bigg
( \frac{-|\vec k|^2}{2}+\frac{|\vec k|}{2}+\frac{1}{2}\sum_{j=1}^r
\bigg (2\beta_j-2 \sum_{i=1}^{j-1}k_i+n_j-k_j-1 \bigg
)(n_j-k_j)\bigg)+ \cdots,
\end{multline*}
we find that  \[M_{\vec n}^{(2)}(x)= x^{|\vec n|}+ \delta _{\vec{n}}
x^{|\vec n|-1}+ \cdots , \] 
where
\begin{multline*}
\delta_{\vec n}= \sum_{k_1=0}^{n_1} \cdots \sum_{k_r=0}^{n_r}
\binom{n_1}{k_1} \cdots \binom {n_r}{k_r}\prod _{j=1}^r \frac
{c^{n_j-k_j}}{(c-1)^{n_j}}\\
 \times (-1)^{|\vec k|}\bigg \{
\frac{-|\vec k|^2}{2}+\frac{|\vec k|}{2}+\frac{1}{2}\sum_{j=1}^r
\bigg (2\beta_j-2 \sum_{i=1}^{j-1}k_i+n_j-k_j-1 \bigg
)(n_j-k_j)\bigg \},
\end{multline*}
which leads, with some calculations using the formula
 \[ \sum _{j=2}^r \bigg( k_j \sum_{i=1}^{j-1}k_i \bigg )= \frac{|\vec k|^2}{2}-\frac{1}{2} \sum_{j=1}^r k_j^2, \] 
and the binomial theorem, to
\[ \delta_{\vec n}= \frac{1}{c-1}\sum_{j=1}^r n_j \bigg (n_j+\beta_j-1+
\sum_{i=1}^{j-1}n_i \bigg )+\sum_{j=1}^r \bigg
(\frac{n_j^2}{2}+\frac{n_j}{2}(2\beta_j-1) \bigg ) . \] 
If we compare the coefficients of $x^{|\vec n|}$ in (\ref{eq:recur}), then
$ b_{\vec n, k}= \delta_{\vec n}- \delta_{\vec n+ \vec e_k}$, which
for the multiple Meixner polynomials of the second kind gives
\begin{equation}\label{Meix2:coef2}
  b_{\vec n,k}= \frac{c}{1-c}(n_k+\beta_k)+ \frac{|\vec n|}{1-c} .
\end{equation}
For the recurrence coefficients $a_{\vec n,j}$ we can again use \cite[Thm.~3.2, Eq. (3.12)]{WVA}
\[ \frac{a_{\vec n,j}}{a_{\vec n+\vec e_i,j}} =\frac{b_{\vec n-\vec e_j,j}-b_{\vec n-\vec e_j,i}}{b_{\vec n,j}-b_{\vec
n,i}} . \] 
Using (\ref{Meix2:coef2}) we find that
\[  a_{\vec n,j}= \frac {n_j+\beta_j-n_i-\beta_i+1}{n_j+\beta_j-n_i-\beta_i} a_{\vec n-\vec e_i,j} . \] 
Using this formula recursively to go from $\vec{n}$ to $n_j \vec{e}_j$ by decreasing each of the indices (except the $j$th index) by one, we
find
\[ a_{\vec n,j}= \prod_{i=1,i\neq j}^r\frac{n_j+\beta_j-\beta_i}{n_j+\beta_j-n_i-\beta_i} \ a_{n_j\vec{e}_j,j}, \]
where $a_{n_j \vec e_j,j}$ is the recurrence coefficients $a_{n_j}^2$ for the Meixner polynomials $M_n(x;\beta_j,c)$, which is equal
to $\frac{c n_j}{(1-c)^2}(n_j+\beta_j-1)$. This gives
\[ a_{\vec n,j}= cn_j\frac{(n_j+\beta_j-1)}{(1-c)^2} \prod_{i\neq j}^r \frac{n_j+\beta_j-\beta_i}{n_j+\beta_j-n_i-\beta_i} . \]
\end{proof}

\section{Non-Hermitian oscillator Hamiltonians}  \label{oscillator}
In \cite{MVZ} Miki et al. showed how a set of $r$ non-Hermitian oscillator Hamiltonians can be obtained which are simultaneously diagonalizable,
with real spectra  and common eigenstates which are expressed in terms of multiple Charlier polynomials.
This idea was further extended in \cite{MTVZ} to construct $r$ non-Hermitian oscillator Hamiltonians in $r$ dimensions which are
simultaneously diagonalizable and for which the spectra are real and the common eigenstates are expressed in terms of multiple Meixner polynomials
of the first kind. Here we will construct $r$ non-Hermitian oscillator Hamiltonians for which the common eigenstates are expressed in terms
of multiple Meixner polynomials of the second kind.

The construction uses a Hilbert space $\mathcal{H}$ with an orthonormal basis $\{|n\rangle, n=1,2,3,\ldots\}$ such that
\[   \langle n|m \rangle = \delta_{n,m}, \]
and on this Hilbert space one has two operators $b$ and $b^+$ with the property
\[    \begin{cases}
    b|n\rangle = \sqrt{n} |n-1\rangle, & \textrm{annihilation operator}, \\
    b^+|n\rangle = \sqrt{n+1} |n+1\rangle, & \textrm{creation operator}.  
       \end{cases}  \]
The ground state $|0\rangle$ has the property $b|0\rangle =0$ and all the states $|n\rangle$ are obtained from the ground state by repeated
application of the operator $b^+$ in the sense that $|n\rangle = (b^+)^n|0\rangle /\sqrt{n!}$.
The annihilation and creation operators $b$ and $b^+$ satisfy the commutation relations
\[   [b,b]=[b^+,b^+]=0, \qquad [b,b^+]=1,  \]
and generate the Lie algebra $W$ with generators $\{b,b^+,1\}$. If we take $r$ copies $W_1,\ldots,W_r$ then we can construct
the algebra
\[  W(r) = \bigoplus_{i=1}^r W_i  \]
where the annihilating operators $b_1,\ldots,b_r$ are commuting and the creation operators $b_1^+,\ldots,b_r^+$ are also commuting, and
\[   [b_i,b_j^+] = \delta_{i,j}.  \]
We denote by $|n_1,\ldots,n_r\rangle$ the standard basis in the $r$-dimensional Hilbert space $\mathcal{H} \otimes \cdots \otimes \mathcal{H}$, 
and we have
\[   b_i|n_1,\ldots,n_i,\ldots,n_r\rangle = \sqrt{n_i} |n_1,\ldots,n_i-1,\ldots,n_r\rangle,  \]
\[   b_i^+|n_1,\ldots,n_i,\ldots,n_r\rangle = \sqrt{n_i+1} |n_1,\ldots,n_i+1,\ldots,n_r\rangle.  \]
Observe that
\[   b_i^+b_i|n_1,\ldots,n_i,\ldots,n_r\rangle = n_i |n_1\,\ldots,n_i,\ldots,n_r\rangle, \qquad 1 \leq i \leq r, \]
where $n_i \in \mathbb{N}$.

\subsection{Multiple Meixner polynomials of the first kind}  \label{MM1}
Miki et al. considered in \cite{MTVZ} the Hamiltonians
\[   H_i = b_i + \sum_{k=1}^r \frac{b_k^+b_k}{1-c_k} + \frac{c_i}{1-c_i} \left( \beta + \sum_{k=1}^r b_k^+b_k \right)
       + \left(\beta + \sum_{k=1}^r b_k^+b_k \right) \sum_{j=1}^r \frac{c_j}{(1-c_j)^2} b_j^+  \]
for $j=1,\ldots,r$,
and they showed that the states
\[   |x\rangle = N_x^{(1)} \sum_{n_1=0}^\infty \cdots \sum_{n_r=0}^\infty \frac{M_{\vec{n}}^{(1)}(x)}{\sqrt{n_1!\cdots n_r!}}
    |n_1,\ldots,n_r\rangle, \]
where $N_x^{(1)}$ are normalizing constants, are common eigenstates for the operators $H_1,\ldots,H_r$:
\[    H_i|x\rangle = x |x\rangle, \qquad 1 \leq i \leq r. \]
The normalizing constants are such that $\langle x|x\rangle = 1$, hence
\[   [N_x^{(1)}]^2 = \left( \sum_{n_1=0}^\infty \cdots \sum_{n_r=0}^\infty \frac{[M_{\vec{n}}^{(1)}(x)]^2}{n_1!\cdots n_r!} \right)^{-1}, \]
however Miki et al.\ did not show that the infinite sum is finite. In fact this sum diverges for every $x \in \mathbb{R}$, as we show in the appendix.
We propose to renormalize the standard basis $|n_1,\ldots,n_r\rangle$ and to take their norm to be
\[   \langle n_1,\ldots,n_r|n_1,\ldots,n_r\rangle = \frac{t_1^{n_1}\cdots t_r^{n_r}}{(\beta)_{|\vec{n}|}}, \]
where $(t_1,\ldots,t_r) \in \mathbb{R}_+^r$ satisfies
\[    \sum_{j=1}^r \frac{c_j^2}{(1-c_j)^2} t_j < 1, \]
and at least one of the $t_i$ $(1 \leq i \leq r)$ satisfies $t_i > (1-c_i)^2$. For instance, if $c_1+\cdots+c_r <1$ then one can take 
$t_j = (1-c_j)^2/c_j$ for all $j \in \{1,2,\ldots,r\}$.
Then the norm of the eigenstate in the weighted
$r$-dimensional Hilbert space is
\[   [\hat{N}_x^{(1)}]^{-1} = \left( \sum_{n_1=0}^\infty \cdots \sum_{n_r=0}^\infty \frac{[M_{\vec{n}}^{(1)}(x)]^2}{ n_1!\cdots n_r!(\beta)_{|\vec{n}|}} 
   t_1^{n_1} \cdots t_r^{n_r} \right)^{1/2}  \]
and this is finite if and only if $x \in \mathbb{N}=\{0,1,2,\ldots\}$ (see Theorem \ref{thm:App1} in the appendix).
This way there are countably many eigenstates for the eigenvalues $x \in \mathbb{N}$.

\subsection{Multiple Meixner polynomials of the second kind}   \label{MM2}
We will now construct Hamiltonians for which the common eigenstates are expressed in terms of multiple Meixner polynomials of the second kind.
This is a new construction but it follows similar ideas as in \cite{MTVZ}.
Let $0 < c < 1$ and $\beta_1,\ldots,\beta_r>0$ be such that $\beta_i-\beta_j \notin \mathbb{Z}$ whenever $i \neq j$. 
We first introduce $r$ operators $B_j$ by
\[  B_j = \prod_{\ell=1, \ell \neq j}^r (b_j^+b_j + \beta_j - b_\ell^+b_\ell-\beta_\ell), \qquad 1 \leq j \leq r, \]
which act on the basis $|n_1,\ldots,n_r\rangle$ as
\[    B_j |n_1,\ldots,n_r\rangle = \prod_{\ell=1, \ell \neq j}^r (n_j + \beta_j - n_\ell-\beta_\ell) |n_1,\ldots,n_r\rangle. \]
Hence the $|n_1,\ldots,n_r\rangle$ are eigenstates and the corresponding eigenvalues are nonzero because $\beta_j-\beta_\ell \notin \mathbb{Z}$ 
whenever $\ell \neq j$.
Hence $B_1,\ldots,B_r$ are invertible. We define the Hamiltonians 
\begin{multline}  \label{H2}
  H_i = b_i + \frac{c}{1-c} (b_i^+b_i+\beta_i) + \frac{1}{1-c} \sum_{j=1}^r b_j^+b_j \\
      + \frac{c}{(1-c)^2} \sum_{j=1}^r (b_j^+b_j+\beta_j) \prod_{k=1,k\neq j}^r (b_j^+b_j+\beta_j-\beta_k) B_j^{-1} b_j^+, \qquad 1 \leq i \leq r,
\end{multline}
and the states
\[   |x\rangle = N_x^{(2)} \sum_{n_1=0}^\infty \cdots \sum_{n_r=0}^\infty \frac{M_{\vec{n}}^{(2)}(x)}{\sqrt{n_1!\cdots n_r!}} |n_1,\ldots,n_r\rangle. \]
Note that each $H_i$ contains an annihilation operator $b_i$ and a number operator $b_i^+b_i$, but the remaining terms are independent of $i$
and involve the total number operator $\sum_{j=1}^r b_j^+b_j$ and all the creation operators $b_1^+,\ldots,b_r^+$. The interpretation is that there are $r$ different kinds of particles and $H_i$ only annihilates particles of the $i$th kind and creates particles of all kinds, independent of $i$ but
with a certain interaction which can be tuned using the parameters $\beta_1,\ldots,\beta_r$ and $c$.   
We normalize the standard basis such that
\[    \langle n_1,\ldots,n_r|n_1,\ldots,n_r\rangle = \frac{t_1^{n_1}\cdots t_r^{n_r}}{(\beta_1)_{n_1}\cdots(\beta_r)_{n_r}}, \]
where $(t_1,\ldots,t_r) \in \mathbb{R}_+^r$ satisfies
\[    t_j < \frac{(1-c)^2}{c^2}, \qquad \forall j \in \{1,2,\ldots,r\} \]
and at least one of the $t_i$ satisfies $t_i  > (1-c)^2$. One can take, for instance, $t_j=(1-c)^2/c$ for all $j \in \{1,2,\ldots,r\}$.
Then the normalizing constants satisfy
\[   [N_x^{(2)}]^2 = \left( \sum_{n_1=0}^\infty \cdots \sum_{n_r=0}^\infty \frac{[M_{\vec{n}}^{(2)}(x)]^2}
    {n_1!\cdots n_r!(\beta_1)_{n_1}\cdots(\beta_r)_{n_r}} t_1^{n_1}\cdots t_r^{n_r} \right)^{-1}.  \]
It is shown in the appendix (Theorem \ref{thm:App2}) that this multiple sum is finite if and only if $x \in \mathbb{N} = \{0,1,2,\ldots\}$, hence
this gives a countable number of eigenstates in the weighted $r$-dimensional Hilbert space.

\section{Further work}
We have given some properties of multiple Meixner polynomials and have shown how these polynomials can be used to construct non-Hermitian oscillator Hamiltonians in $r$ dimensions with a countable number of eigenvalues and eigenstates. 
Other families of multiple orthogonal polynomials, such as multiple Hahn polynomials or multiple Krawtchouk polynomials, may also
enable the construction of Hamiltonians for which the eigenstates can explicitly be given.

\ack 
We thank the referees for suggesting a number of improvements.
This research was supported by FWO grant G.0934.13, KU Leuven research grant OT/12/073 and the Belgian Interuniversity Attraction Pole P7/18.
F.N. is supported by a VLIR/UOS scholarship. This work was initiated in the Ph.D. dissertation of F.N. at KU Leuven and he
thanks the Department of Mathematics at KU Leuven for their hospitality.  

\appendix

\section*{Appendix}
\setcounter{section}{1}

If $(p_n)_{n\in \mathbb{N}}$ is a sequence of orthonormal polynomials on the real line for a positive measure
$\mu$ for which the moment problem is determinate,
\[   \int_{\mathbb{R}} p_n(x)p_m(x)\, d\mu(x) = \delta_{m,n}, \]
then it is well known that
\begin{equation}  \label{app1}
   \sum_{n=0}^\infty [p_n(x)]^2 = \frac{1}{\mu(\{x\})}, 
\end{equation}
so that this infinite sum is finite if and only if $x$ is a mass point of the measure $\mu$, i.e., $x$ belongs to the support of the discrete
part $\mu_d$ of the measure $\mu$ \cite[Theorem 2.5.6]{Ismail}. The orthogonality measure for Meixner polynomials is the negative binomial
distribution (or Pascal distribution) which is a discrete measure supported on the integers $\mathbb{N} = \{0,1,2,\ldots\}$. Following
\cite[\S 9.10]{Koekoek} the Meixner polynomials are given by
\begin{equation}  \label{app2}
   \tilde{M}_n(x;\beta,c) = {}_2F_1 \left( \begin{array}{c} -n,-x \\ \beta \end{array} ; 1-\frac{1}{c} \right), 
\end{equation}
and then the orthogonality for the Meixner polynomials is
\begin{equation}  \label{app3}
  (1-c)^\beta \sum_{k=0}^\infty \tilde{M}_n(k;\beta,c)\tilde{M}_m(k;\beta,c) \frac{(\beta)_k c^k}{k!} = \frac{n!}{c^n (\beta)_n} \delta_{m,n}.
\end{equation}
The orthonormal polynomials $p_n(x)$ and the monic polynomials $M_n(x)$ are given by
\[  p_n(x) = \tilde{M}_n(x;\beta,c) \sqrt{ \frac{(\beta)_n c^n}{n!}}, \quad
    M_n(x) = (\beta)_n \left( \frac{c}{c-1} \right)^n \tilde{M}_n(x);\beta,c), \]
so that \eqref{app1} becomes
\begin{equation} \label{app4}
    (1-c)^\beta \sum_{n=0}^\infty \frac{[M_n(x)]^2}{n! (\beta)_n} \frac{(c-1)^{2n}}{c^n} 
     = \begin{cases}  \frac{k!}{c^k (\beta)_k}, & x=k \in \mathbb{N} = \{0,1,2,\ldots \}, \\
                       + \infty, & x \in \mathbb{R} \setminus \mathbb{N}.  \end{cases}
\end{equation}
Observe that \eqref{app2} implies that $\tilde{M}_n(k;\beta,c) = \tilde{M}_k(n;\beta,c)$, hence the
orthogonality \eqref{app3} also gives
\begin{equation}  \label{app5}
  (1-c)^\beta \sum_{n=0}^\infty \frac{M_n(k)M_n(\ell)}{n!(\beta)_n} \frac{(c-1)^{2n}}{c^n} = \frac{k!}{c^k (\beta)_k} \delta_{k,\ell}, 
\end{equation}
for which the case $k=\ell$ corresponds to \eqref{app4}. In fact more is true: the generating function \eqref{GenMeix1} for the monic 
Meixner polynomials is
\[    \sum_{n=0}^\infty M_n(x) \frac{t^n}{n!} = \left( 1+ \frac{t}{1-c} \right)^x \left(1+ \frac{c}{1-c}t \right)^{-x-\beta}. \]
When $x \notin \mathbb{N}$ we see that the singularity closest to the origin is at $t=-(1-c)$, hence the radius of convergence is $R=1-c$.
This means that
\[   \limsup_{n \to \infty} \left( \frac{|M_n(x)|}{n!} \right)^{1/n} = \frac{1}{1-c}, \]
and hence
\[   \limsup_{n \to \infty} \left( \frac{|M_n(x)|^2}{n!(\beta)_n} \right)^{1/n} 
   = \limsup_{n \to \infty} \left( \frac{|M_n(x)|^2}{(n!)^2} \right)^{1/n} \left( \frac{n!}{(\beta)_n} \right)^{1/n} = \frac{1}{(1-c)^2}, \]
so that for $x \notin \mathbb{N}$
\begin{equation}  \label{app6}
   \sum_{n=0}^\infty  \frac{[M_n(x)]^2}{n! (\beta)_n} t^n  \quad \begin{cases}  < \infty & \textrm{if } 0 \leq t < (1-c)^2, \\
                                                                                  = \infty & \textrm{if } t > (1-c)^2. 
                                                                   \end{cases}
\end{equation}
The case $t=(1-c)^2/c > (1-c)^2$ corresponds to the divergent part of \eqref{app4}. When $x \in \mathbb{N}$ the first singularity in the generating function
is at $t=-(1-c)/c$, hence 
\[   \limsup_{n \to \infty} \left( \frac{|M_n(x)|}{n!} \right)^{1/n} = \frac{c}{1-c}, \]
and therefore
\[   \limsup_{n \to \infty} \left( \frac{|M_n(x)|^2}{n!(\beta)_n} \right)^{1/n} = \frac{c^2}{(1-c)^2}.  \]
This means that for $x \in \mathbb{N}$
\begin{equation}  \label{app7}
   \sum_{n=0}^\infty  \frac{[M_n(x)]^2}{n! (\beta)_n} t^n  \quad \begin{cases}  < \infty & \textrm{if } 0 \leq t < (1-c)^2/c^2, \\
                                                                                  = \infty & \textrm{if } t > (1-c)^2/c^2. 
                                                                   \end{cases}. 
\end{equation}
The case $t=(1-c)^2/c < (1-c)^2/c^2$ corresponds to the convergent part of \eqref{app4}.

We now prove two similar results for the multiple Meixner polynomials, which are needed in Section \ref{oscillator}

\begin{theorem} \label{thm:App1}
Let 
\[   A_1 = \{ (t_1,\ldots,t_r) \in \mathbb{R}_+^r : \left( \frac{c_1}{1-c_1} \right)^2 t_1 + \cdots + \left( \frac{c_r}{1-c_r} \right)^2 t_r < 1 \}, \]
and
\[   B_i = \{ (t_1,\ldots,t_r) \in \mathbb{R}_+^r : t_i > (1-c_i)^2 \}. \]
For multiple Meixner polynomials of the first kind one has for $(t_1,\ldots,t_r) \in A_1 \cap \bigcup_{i=1}^r B_i$
\[  \sum_{n_1=0}^\infty \cdots \sum_{n_r=0}^\infty \frac{[M_{\vec{n}}^{(1)}(x)]^2}{n_1!\cdots n_r! (\beta)_{|\vec{n}|}} 
     t_1^{n_1} \cdots t_r^{n_r}  \quad \begin{cases} < \infty, &  \textrm{if } x \in \mathbb{N} = \{0,1,2,\ldots\} ,\\
                                               = \infty & \textrm{if } x \in \mathbb{R} \setminus \mathbb{N}. 
                                 \end{cases} \] 
\end{theorem}    

\begin{proof}
Suppose that $x \in \mathbb{R} \setminus \mathbb{N}$ and $t_i > (1-c_i)^2$, then 
\[   \sum_{n_1=0}^\infty \cdots \sum_{n_r=0}^\infty \frac{[M_{\vec{n}}^{(1)}(x)]^2}{n_1!\cdots n_r! (\beta)_{|\vec{n}|}} 
     t_1^{n_1} \cdots t_r^{n_r} \geq \sum_{n_i=0}^\infty \frac{[M_{(0,\ldots,0,n_i,0,\ldots,0)}^{(1)}(x)]^2}{n_i! (\beta)_{n_i}} t_i^{n_i}, \]
and this single sum diverges since $M_{(0,\ldots,0,n_i,0,\ldots,0)}^{(1)}(x)$ is the monic Meixner polynomial $M_n(x;\beta,c_i)$, see \eqref{app6}.
This shows the divergent part of the theorem.

Suppose that $x \in \mathbb{N}$, say $x=k$, then the multiple sum \eqref{explicit 1} only contains a finite number of terms since
$(-k)_{|\vec{k}|} = 0$ whenever $|\vec{k}| > k$. This means that 
\[  M_{\vec{n}}^{(1)}(k) = \prod_{j=1}^r \frac{c_j^{n_j}}{(c_j-1)^{n_j}} 
 \sum_{|\vec{k}| \leq k} \binom{n_1}{k_1}\cdots \binom{n_r}{k_r} (-k)_{|\vec{k}|} (\beta+k)_{|\vec{n}|-|\vec{k}|}
    \prod_{j=1}^r c_j^{-k_j}.  \] 
A simple estimation, using $|(-k)_{|\vec{k}|}| \leq k!$ and $(\beta+k)_{|\vec{n}|-|\vec{k}|} = (\beta)_{|\vec{n}|}(\beta+|\vec{n}|)_{k-|\vec{k}|}/(\beta)_k$ gives
\begin{align*}
   |M_{\vec{n}}^{(1)}(k)| &\leq \prod_{j=1}^r \frac{c_j^{n_j}}{(1-c_j)^{n_j}} k! (\beta)_{|\vec{n}|} 
 \sum_{|\vec{k}| \leq k} \binom{n_1}{k_1}\cdots \binom{n_r}{k_r}  \frac{(\beta+|\vec{n}|)_{k-|\vec{k}|}}{(\beta)_k}
    \prod_{j=1}^r c_j^{-k_j} \\
   & =  \prod_{j=1}^r \frac{c_j^{n_j}}{(1-c_j)^{n_j}} k! (\beta)_{|\vec{n}|} P(n_1,\ldots,n_r),
\end{align*}
where $P(n_1,\ldots,n_r)$ is a polynomial in the variables $n_1,\ldots,n_r$. This means that
\begin{multline*}  
 \sum_{n_1=0}^\infty \cdots \sum_{n_r=0}^\infty \frac{[M_{\vec{n}}^{(1)}(k)]^2}{n_1!\cdots n_r! (\beta)_{|\vec{n}|}} 
     t_1^{n_1} \cdots t_r^{n_r} \\
  \leq  (k!)^2 \sum_{n_1=0}^\infty \cdots \sum_{n_r=0}^\infty [P(n_1,\ldots,n_r)]^2 
   \frac{c_1^{2n_1} \cdots c_r^{2n_r}}{(1-c_1)^{2n_1}\cdots(1-c_r)^{2n_r}} \frac{(\beta)_{|\vec{n}|}}{n_1!\cdots n_r!} t_1^{n_1} \cdots t_r^{n_r}.  
\end{multline*}
The latter sum converges for $(t_1,\ldots,t_r) \in A_1$ (see Lemma \ref{lemma:Meix} and Abel's lemma \cite[Prop.~2.3.4 on p.~87]{Krantz}). 
This proves the convergent part of the theorem.
\end{proof}

\begin{theorem}  \label{thm:App2}
Let 
\[   A_2 = \{ (t_1,\ldots,t_r) \in \mathbb{R}_+^r : t_i < \left( \frac{1-c}{c} \right)^2,\  1 \leq i \leq r\} \]
and 
\[   B_i = \{ (t_1,\ldots,t_r) \in \mathbb{R}_+^r : t_i > (1-c)^2 \}. \]
For multiple Meixner polynomials of the second kind one has for $(t_1,\ldots,t_r) \in A_2 \cap \bigcup_{i=1}^r B_i$
\[  \sum_{n_1=0}^\infty \cdots \sum_{n_r=0}^\infty \frac{[M_{\vec{n}}^{(2)}(x)]^2}{n_1!\cdots n_r!(\beta_1)_{n_1}\cdots (\beta_r)_{n_r}} 
     t_1^{n_1}\cdots t_r^{n_r} \quad 
  \begin{cases}   < \infty, & \textrm{if } x \in \mathbb{N} = \{0,1,2,\ldots\}, \\
                   = \infty, & \textrm{if } x \in \mathbb{R} \setminus \mathbb{N}.
   \end{cases}  \]
\end{theorem} 

\begin{proof}
Suppose that $x \in \mathbb{R} \setminus \mathbb{N}$ and $t_i > (1-c)^2$, then 
\[   \sum_{n_1=0}^\infty \cdots \sum_{n_r=0}^\infty \frac{[M_{\vec{n}}^{(2)}(x)]^2}{n_1!\cdots n_r! (\beta_1)_{n_1}\cdots (\beta_r)_{n_r}} 
     t_1^{n_1} \cdots t_r^{n_r} \geq \sum_{n_i=0}^\infty \frac{[M_{(0,\ldots,0,n_i,0,\ldots,0)}^{(2)}(x)]^2}{n_i! (\beta_i)_{n_i}} t_i^{n_i}, \]
and this single sum diverges since $M_{(0,\ldots,0,n_i,0,\ldots,0)}^{(2)}(x)$ is the monic Meixner polynomial $M_n(x;\beta_i,c)$, see \eqref{app6}.
This shows the divergent part of the theorem.

Suppose that $x \in \mathbb{N}$, say $x=k$, then the multiple sum \eqref{explicit 2} only contains a finite number of terms since
$(-k)_{|\vec{k}|} = 0$ whenever $|\vec{k}| > k$. This means that 
\[  M_{\vec{n}}^{(2)}(k) =  \frac{c^{|\vec{n}|}}{(c-1)^{|\vec{n}|}} 
 \sum_{|\vec{k}| \leq k} \binom{n_1}{k_1}\cdots \binom{n_r}{k_r} (-k)_{|\vec{k}|} c^{-|\vec{k}|} 
   \prod_{j=1}^r \bigl(\beta_j+k-\sum_{i=1}^{j-1} k_i \bigr)_{n_j-k_j}  .  \] 
A simple estimation, using $|(-k)_{|\vec{k}|}| \leq k!$ and $(\beta_j+k-\ell)_{n_j-k_j} = (\beta_j)_{n_j}(\beta_j+n_j)_{k-k_j-\ell}/(\beta_j)_{k-\ell}$, gives
\begin{align*}
   |M_{\vec{n}}^{(2)}(k)| &\leq \frac{c^{|\vec{n}|}}{(1-c)^{|\vec{n}|}} k! \prod_{j=1}^r (\beta_j)_{n_j} 
 \sum_{|\vec{k}| \leq k} \binom{n_1}{k_1}\cdots \binom{n_r}{k_r}  c^{-|\vec{k}|} \prod_{j=1}^r 
  \frac{(\beta_j+n_j)_{k-\sum_{i=1}^j k_i}}{(\beta_j)_{k-\sum_{i=1}^{j-1}k_i}} \\
   & =  \frac{c^{|\vec{n}|}}{(1-c)^{|\vec{n}|}} k! \prod_{j=1}^r (\beta_j)_{n_j} Q(n_1,\ldots,n_r),
\end{align*}
where $Q(n_1,\ldots,n_r)$ is a polynomial in the variables $n_1,\ldots,n_r$. This means that
\begin{multline*}  
 \sum_{n_1=0}^\infty \cdots \sum_{n_r=0}^\infty \frac{[M_{\vec{n}}^{(2)}(k)]^2}{n_1!\cdots n_r! (\beta_1)_{n_1}\cdots (\beta_r)_{n_r}} 
     t_1^{n_1} \cdots t_r^{n_r} \\
  \leq  (k!)^2 \sum_{n_1=0}^\infty \cdots \sum_{n_r=0}^\infty [Q(n_1,\ldots,n_r)]^2 
   \frac{c^{2|\vec{n}|}}{(1-c)^{2|\vec{n}|}} \frac{(\beta_1)_{n_1}\cdots(\beta_r)_{n_r}}{n_1!\cdots n_r!} t_1^{n_1} \cdots t_r^{n_r}.  
\end{multline*}
The latter sum converges for $(t_1,\ldots,t_r) \in A_2$ since the multiple sum without the polynomial $Q$ factors into a product of
$r$ binomial series and the multiple sum with the polynomial $Q$ still converges by Abel's lemma \cite[Prop.~2.3.4 on p.~87]{Krantz}.  
This proves the convergent part of the theorem.
\end{proof}

Similar results are true for Charlier polynomials and multiple Charlier polynomials. These were considered in \cite{MVZ}.
The Charlier polynomials $C_n^{(a)}$ are orthogonal polynomials for the Poisson distribution
\[   e^{-a} \sum_{k=0}^\infty C_n^{(a)}(k)C_m^{(a)}(k) \frac{a^k}{k!} =  a^n n! \delta_{m,n},  \]
where $a>0$. They are given by
\[  C_n^{(a)}(x) = (-a)^n {}_2F_0 \left( \left. \begin{array}{c} -n,-x \\ - \end{array} \right| - \frac{1}{a} \right),  \]
from which we find that $C_n^{(a)}(k)/(-a)^n = C_k^{(a)}(n)/(-a)^k$. Hence the orthogonality also gives
\[     \sum_{k=0}^\infty C_k^{(a)}(n)C_k^{(a)}(m) \frac{1}{a^kk!} = e^a a^{-n} n! \delta_{m,n}. \]
The orthonormal Charlier polynomials are $p_n(x)= C_n^{(a)}(x)/\sqrt{a^n n!}$, so that the last sum for $n=m$ becomes
\[   \sum_{k=0}^\infty  [p_k(n)]^2 = e^{a} a^{-n} n!, \]
which indeed corresponds to \eqref{app1}. In fact more is true: the generating function for Charlier polynomials is
\cite[\S VI.1]{Chihara}
\[    \sum_{n=0}^\infty C_n^{(a)}(x) \frac{t^n}{n!} = (1+t)^x e^{-at}. \]
When $x \notin \mathbb{N}$ we see that this function has a singularity at $t=-1$, hence the radius of convergence is $R=1$, which implies that
\[  \limsup_{n \to \infty} \left( \frac{|C_n^{(a)}(x)|}{n!} \right)^{1/n} = 1. \]
But then
\[  \limsup_{n \to \infty} \left( \frac{|C_n^{(a)}(x)|^2}{n!} \right)^{1/n} 
    =  \limsup_{n\to \infty}  \left( \frac{|C_n^{(a)}(x)|^2}{(n!)^2} \right)^{1/n} (n!)^{1/n} = \infty , \]
so that for $x \notin \mathbb{N}$ 
\[   \sum_{n=0}^\infty \frac{[C_n^{(a)}(x)]^2}{n!} t^n = \infty   \]
for every $t>0$. When $x \in \mathbb{N}$ the generating function has no singularities and hence the radius of convergence is $\infty$.
One has 
\[   C_n^{(a)}(k) = (-a)^n \sum_{j=0}^k \binom{n}{j} (-k)_j a^{-j} \]
so that
\[   |C_n^{(a)}(k)| \leq a^n k! \sum_{j=0}^k \binom{n}{j} a^{-j} \leq  k! (a+1)^n, \]
giving
\[   \sum_{n=0}^\infty \frac{[C_n^{(a)}(k)]^2}{n!} t^n \leq (k!)^2 \sum_{n=0}^\infty \frac{(a+1)^{2n}}{n!} t^n = e^{(a+1)^2t} \]
so that this series converges for every $t >0$. Hence one has for every $t>0$
\[   \sum_{n=0}^\infty  \frac{[C_n^{(a)}(x)]^2}{n!} t^n \quad \begin{cases} < \infty, & \textrm{if } x \in \mathbb{N} = \{0,1,2,\ldots\}, \\
                                                                            = \infty, & \textrm{if } x \in \mathbb{R} \setminus \mathbb{N}.
                                          \end{cases}. \]   
Multiple Charlier polynomials are defined by
\[  C_{\vec{n}}(x) = (-1)^{|\vec{n}|} \sum_{k_1=0}^{n_1} \cdots \sum_{k_r=0}^{n_r} \binom{n_1}{k_1} \cdots \binom{n_r}{k_r} 
     (-x)_{k_1+\cdots+k_r} a_1^{n_1-k_1} \cdots a_r^{n_r-k_r}, \]
where $a_i >0$ for $1 \leq i \leq r$ and $a_i \neq a_j$ whenever $i \neq j$. 
The square summability for these multiple Charlier polynomials, needed for the normalization of the eigenfunctions in \cite{MVZ}, is given
by

\begin{theorem}  \label{thm:App3}
For multiple Charlier polynomials one has for all $(t_1,\ldots,t_r) \in \mathbb{R}_+^r$
\[   \sum_{n_1=0}^\infty \cdots \sum_{n_r=0}^\infty \frac{[C_{\vec{n}}(x)]^2}{n_1!\cdots n_r!} t_1^{n_1} \cdots t_r^{n_r} \quad 
    \begin{cases}  < \infty, & \textrm{if } x \in \mathbb{N} = \{0,1,2,\ldots\}, \\
                   = \infty, & \textrm{if } x \in \mathbb{R} \setminus \mathbb{N}.
     \end{cases}  \]
\end{theorem}
The proof is similar (and in fact a bit easier) to the proofs before. The choice $t_1=\cdots=t_r=1$ shows that the eigenstates in \cite{MVZ}
are indeed those corresponding to the eigenvalues $x = k \in \mathbb{N}$.

\end{document}